\documentclass[12pt]{amsart}

\usepackage{amssymb}

\newcommand{\excise}[1]{}

\newtheorem{thm}{Theorem}[section]
\newtheorem{lemma}[thm]{Lemma}
\newtheorem{cor}[thm]{Corollary}
\newtheorem{prop}[thm]{Proposition}

\theoremstyle{definition}
\newtheorem{example}[thm]{Example}
\newtheorem{remark}[thm]{Remark}
\newtheorem{defn}[thm]{Definition}

\numberwithin{equation}{section}

\newcommand{\ring}[1]{\ensuremath{\mathbb{#1}}}

\newcommand\0{\mathbf{0}}
\newcommand\CC{\ring{C}}
\newcommand\NN{\ring{N}}
\newcommand\QQ{\ring{Q}}
\newcommand\RR{\ring{R}}
\newcommand\ZZ{\ring{Z}}
\newcommand\kk{\Bbbk}
\newcommand\pp{{\mathfrak p}}
\newcommand\uu{{\mathbf u}}
\newcommand\xx{{\mathbf x}}
\newcommand\cP{{\mathcal P}}
\newcommand\cW{{\mathcal W}}
\newcommand\oA{\hspace{.35ex}\ol{\hspace{-.35ex}A\hspace{-.2ex}}\hspace{.2ex}}
\newcommand\oQ{\hspace{.15ex}\ol{\hspace{-.15ex}Q\hspace{-.25ex}}\hspace{.25ex}}
\renewcommand\aa{{\mathbf a}}

\renewcommand\th{\mathrm{th}}

\newcommand\til{\mathord\sim}

\newcommand\onto{\twoheadrightarrow}
\newcommand\minus{\smallsetminus}
\renewcommand\iff{\Leftrightarrow}
\renewcommand\implies{\Rightarrow}

\def\ol#1{{\overline {#1}}}

\begin{document}

\mbox{}
\vspace{-7ex}
\title{Affine stratifications from finite mis\`ere quotients}
\author{Ezra Miller}
\address{Mathematics Department\\Duke University\\Durham, NC 27708}
\email{http://math.duke.edu/\~{\hspace{-.2ex}}ezra}
\thanks{The author had support from NSF grants DMS-0449102 = DMS-1014112 and DMS-1001437}

\makeatletter\@namedef{subjclassname@2010}{\textup{2010} Mathematics
	Subject Classification}\makeatother
%
\renewcommand{\keywordsname}{Key words}

\date{23 July 2011}

\begin{abstract}
Given a morphism from an affine semigroup $Q$ to an arbitrary
commutative monoid, it is shown that every fiber possesses an affine
stratification: a partition into a finite disjoint union of translates
of normal affine semigroups.  The proof rests on mesoprimary
decomposition of monoid congruences
and a novel list of equivalent conditions characterizing the existence
of an affine stratification.  The motivating consequence of the main
result is a special case of a conjecture due to Guo and the author
on the existence of affine stratifications for (the set of winning
positions~of) any lattice game.  The special case proved here assumes
that the lattice game has finite mis\`ere quotient, in the sense of
Plambeck and Siegel.\vspace{-3.2ex}
\end{abstract}
\maketitle

\section{Introduction}

Lattice games encode finite impartial combinatorial games---and
winning strategies for them---in terms of lattice points in rational
convex polyhedra \cite{latticeGames}; see Section~\ref{s:misere}.  The
concept grew out of Plambeck's theory of mis\`ere quotients
\cite{Pla05}, as developed by Plambeck and Siegel \cite{misereQuots}
(see also Siegel's lecture notes \cite{siegelLectures}, particularly
Figure~7 in \mbox{Lecture}~5 there).  Their purpose was to provide
data structures for recording and computing winning strategies of
combinatorial games, such as octal games, under the mis\`ere play
condition, where the last player to move loses.  In that spirit, Guo
and the author conjectured that the lattice points encoding the
strategy of any lattice game have a particularly well-behaved
presentation, called an affine stratification: a partition into a
finite disjoint union of translates of affine semigroups
\cite[Conjecture~8.9]{latticeGames}.  The conjecture is as far from
true as possible, in general, because lattice games support universal
computation, as shown by Fink \cite{fink11}.
However, the most successful applications of mis\`ere quotients thus
far have occurred when the quotient is finite, because of amenability
to algorithmic computation.  Bridging mis\`ere quotient theory and
lattice games in the case of finite quotients is therefore one of the
primary intents of this note, whose motivating result is the existence
of affine stratifications for lattice games with finite mis\`ere
quotients~(Theorem~\ref{t:affineStrat}).

The proof comes via a more general main result, of independent
interest, on morphisms from affine semigroups to arbitrary commutative
monoids: the fibers of such morphisms possess affine stratifications
(Theorem~\ref{t:strat}).  That result, in turn, follows from two
additional results of independent interest.  The first is a host of
equivalent conditions characterizing the existence of an affine
stratification (Theorem~\ref{t:equiv}).  The second is existence of
mesoprimary decomposition for congruences on finitely generated
\mbox{commutative} monoids, a theory whose development occupies half
of a full-length paper of its own \cite{mesoprimary}.  One of the
motivations for defining mesoprimary decomposition in the first place
was its anticipated relevance to combinatorial game theory, via the
results presented here and via potential further applications toward
the existence of affine stratifications for the class of
\emph{squarefree} lattice games \cite[Corrigendum]{latticeGames}
(which includes all octal games), where the affine stratification
conjecture~remains~open.

Beyond mesoprimary decomposition and other more elementary theory of
commutative monoids, the reasoning in this note involves elementary
polyhedral geometry, including subdivisions and Minkowski sums, as
well as combinatorial commutative algebra of finely graded modules
over affine semigroup rings, particularly filtrations~thereof.

\section{Affine stratifications}\label{s:affine}

The goal is to decompose certain sets of lattice points in polyhedra
in particularly nice ways, following \cite[\S8]{latticeGames}, where
this definition originates.

\begin{defn}\label{d:strat}
An \emph{affine stratification} of a subset $\cW \subseteq \ZZ^d$ is a
finite partition
$$%
  \cW = \biguplus_{i=1}^r W_i
$$
of~$\cW$ into a finite disjoint union of sets~$W_i$, each of which is
a \emph{finitely generated module} for an affine semigroup $A_i
\subseteq \ZZ^d$; that is, $W_i = F_i + A_i$, where $F_i \subseteq
\ZZ^d$ is a finite set and $A_i$ is a finitely generated submonoid of
the free abelian group~$\ZZ^d$ of rank~$d$.
\end{defn}

For the coming sections, it will be helpful to specify, in
Theorem~\ref{t:equiv}, some alternative decompositions equivalent to
affine stratifications.  For that, we need four lemmas as stepping
stones.  In the first, a \emph{normal} affine semigroup is the
intersection of a rational polyhedral cone in~$\RR^d$ with a
sublattice of~$\ZZ^d$.

\begin{lemma}\label{l:sat}
Every affine semigroup $A \subseteq \ZZ^d$ possesses an affine
stratification in which each stratum is a translate $f_i + A_i$ of a
normal affine semigroup $A_i \subseteq \ZZ^d$.
\end{lemma}
\begin{proof}
Let $\oA = \RR_+ A \cap \ZZ A$ denote the \emph{saturation} of~$A$:
the set of lattice points lying in the intersection of the real cone
generated by~$A$ with the group generated by~$A$.  Then $A$ contains a
translate $\aa + \oA$ of its saturation by \cite[Exercise~7.15]{cca}.
Transferring this statement to the language of monoid algebras, the
affine semigroup ring $\CC[A]$ has a $\CC[A]$-submodule
$\xx^\aa\CC[\oA] \subseteq \CC[A]$.  The quotient $M =
\CC[A]/\xx^\aa\CC[\oA]$ is a finitely generated $\ZZ^d$-graded
$\CC[A]$-module.  The module $M$ is therefore \emph{toric}, in the
sense of \cite[Definition~4.5]{rankJumps}, by
\cite[Example~4.7]{rankJumps}.  This means that $M$ has a \emph{toric
filtration} $0 = M_0 \subset M_1 \subset \cdots \subset M_{\ell-1}
\subset M_\ell = M$, in which $M_j/M_{j-1}$ is, for each~$j$, a
$\ZZ^d$-graded translate $\xx^{\aa_j}\CC[F_j]$ of the affine semigroup
ring~$\CC[F_j]$ for some face~$F_j$ of~$A$ and some $\aa_j \in \ZZ^d$.
Transferring this statement back into the language of lattice points,
$$%
  A = (\aa + \oA) \uplus \biguplus_j (\aa_j + F_j)
$$
is a disjoint union of a translated normal affine semigroup $\aa +
\oA$ and a disjoint union of translates $\aa_j + F_j$ of faces of~$A$.
But $\dim M < \dim \CC[A]$, since $\dim \xx^\aa\CC[\oA] = \dim
\CC[A]$, so each face $F_j$ that appears is a proper face of~$A$.
Therefore the proof is done by induction on $\dim \CC[A]$, the case of
dimension~$0$ being trivial, since then $A = \{0\}$.
\end{proof}

In the next lemma, keep in mind that the polyhedra need not be
bounded.

\begin{lemma}\label{l:polyhedra}
Any finite union of (rational) convex polyhedra in~$\RR^d$ can be
expressed as a disjoint union of finitely many sets, each of which is
the relative interior of a (rational) convex polyhedron.
\end{lemma}
\begin{proof}
The polyhedra in the given union~$U$ are defined as intersecions of
finitely many halfspaces.  The totality of all hyperplanes involved
subdivide the ambient space into finitely many closed---but perhaps
unbounded---polyhedral regions.  This union of regions is a
\emph{polyhedral complex} (see \cite[Section~5.1]{ziegler} for the
definition) with finitely many faces, some of which may be unbounded.
By construction, the relative interior of each face is either
contained in~$U$ or disjoint from~$U$, proving the lemma.
\end{proof}

The following will be used in the proofs of both Lemma~\ref{l:union}
and Theorem~\ref{t:equiv}.

\begin{lemma}\label{l:module}
If\/ $\Pi = P + C$ is a rational convex polyhedron in~$\RR^d$,
expressed as the Minkowski sum of a polytope~$P$ and a cone~$C$
\cite[Theorem~1.2]{ziegler}, and\/ $\Pi^\circ$ is its relative
interior, then $\Pi \cap \ZZ^d$ and\/ $\Pi^\circ \cap \ZZ^d$ are
finitely generated modules for $A = C \cap \ZZ^d$.
\end{lemma}
\begin{proof}
Suppose that $\Pi = \bigcap_j \{\xx \in \RR^d \mid \phi_j(\xx) \geq
c_j\}$ for rational linear functions $\phi_j$ and rational
constants~$c_j$.  The case of $\Pi^\circ$ follows from that of~$\Pi$
itself: the lattice points in $\Pi^\circ$ are the same as those in the
closed polyhedron $\Pi_\varepsilon = \bigcap_j \{\xx \in \RR^d \mid
\phi_j(\xx) \geq c_j + \varepsilon\}$ obtained from~$\Pi$ by moving
each of its bounding hyperplanes inward by a small rational distance,
where $\varepsilon$ is less than any nonzero positive value of
$\phi_j$ on~$\ZZ^d$ for all~$j$.  (The rationality of~$\phi_j$
guarantees that $\phi_j(\ZZ^d)$ is a discrete subset of the rational
numbers~$\QQ$.)

Given the closed polyhedron~$\Pi$, consider its \emph{homogenization}
\cite[Section~1.5]{ziegler}: the closure $\ol\Pi$ of the cone over a
copy of $\Pi = \Pi \times \{1\}$ placed at height~$1$ in~$\RR^{d+1} =
\RR^d \times \RR$.  The intersection of~$\ol\Pi$ with the first factor
$\RR^d = \RR^d \times \{0\}$ is the cone~$C$, and $\ol\Pi \cap \ZZ^d =
A$ is a face of the affine semigroup $\oA = \ol\Pi \cap \ZZ^{d+1}$.
The intersection $\Pi \cap \ZZ^d$ is isomorphic, as a module over~$A$,
to the intersection $M = \ol\Pi \cap (\ZZ^d \times \{1\}) = \oA \cap
(\ZZ^d \times \{1\})$ with the copy of~$\ZZ^d$ at height~$1$.  The
result now follows from \cite[Eq.~(1) and Lemma~2.2]{irredres} or
\cite[Theorem~11.13]{cca}, where $M$ is identified as the set of
(exponent vectors of) monomials annihilated by the prime ideal
$\pp_{\oA} \subseteq \CC[\ol\Pi]$ modulo an irreducible monomial ideal
of~$\kk[\ol\Pi]$.  (The prime ideal $\pp_{\oA}$ is the kernel of the
surjection $\CC[\ol\Pi] \onto \CC[\oA]$; this argument is taken from
the proof of \cite[Proposition~2.13]{primDecomp}.)
\end{proof}

\begin{lemma}\label{l:union}
If\/ $\cW \subseteq \ZZ^d$ is a finite union of sets~$W_i$, each a
translate of a normal affine semigroup $A_i = \RR_+ A_i \cap L_i$ for
some sublattice $L_i \subseteq \ZZ^d$, then $\cW$ can be expressed as
such a union in which $L_i = L$ for all~$i$ is a fixed sublattice of
finite index in~$\ZZ^d$.
\end{lemma}
\begin{proof}
Suppose that $A_i = \RR_+ A_i \cap L_i$ is given for all~$i$.  Taking
a direct sum with a complementary sublattice, we may assume that $L_i$
has finite index in~$\ZZ^d$ for all~$i$.  Now set $L = \bigcap_i L_i$.
Then $A_i = \bigcup_{\lambda \in L_i/L} A_i \cap (\lambda + L)$ is a
finite union of sets obtained by intersecting a coset of~$L$
with~$A_i$.  Each such set $A_i \cap (\lambda + L)$ is a finitely
generated module over $A_i \cap L$ by Lemma~\ref{l:module}.  The
desired union is therefore achievable using translates of the normal
affine semigroups $A_i \cap L$.
\end{proof}

\begin{thm}\label{t:equiv}
The following are equivalent for a set $\cW \subseteq \ZZ^d$ of
lattice points.
\begin{enumerate}
\item%
$\cW$ possesses an affine stratification.
\item%
$\cW$ is a finite (not necessarily disjoint) union of sets~$W_i$, each
of which is a finitely generated module for an affine semigroup $A_i
\subseteq \ZZ^d$.
\item%
$\cW$ is a finite (not necessarily disjoint) union of sets~$W_i$, each
of which is a translate $f_i + A_i$ of an affine semigroup $A_i
\subseteq \ZZ^d$.
\item
$\cW$ is a finite (not necessarily disjoint) union of sets~$W_i$, each
of which is a translate $f_i + A_i$ of a normal affine semigroup
$A_i \subseteq\ZZ^d$.
\item%
$\cW$ is a finite disjoint union of sets~$W_i$, each of which is a
translate $f_i + A_i$ of a normal affine semigroup $A_i \subset
\ZZ^d$.
\item%
$\cW$ is a finite disjoint union of sets~$W_i$, each of which is a
translate $f_i + A_i$ of a (not necessarily normal) affine
semigroup $A_i \subset \ZZ^d$.
\end{enumerate}
\end{thm}
\begin{proof}
By definition it follows that 1 $\implies$ 2 $\implies$ 3 and that 5
$\implies$ 6 $\implies$ 1.  It therefore remains only to show that 3
$\implies$ 4 $\implies$ 5.  The first of these implications is
Lemma~\ref{l:sat}.

For the second, begin by choosing the union to satisfy the conclusion
of Lemma~\ref{l:union}.  For each $\lambda \in \ZZ^d/L$, let
$\cW_\lambda = \cW \cap (\lambda + L)$ be the intersection of~$\cW$
with the corresponding coset of~$L$ in~$\ZZ^d$.  Then $\cW_\lambda$ is
the intersection of $\lambda + L$ with the union $U_\lambda$ of those
polyhedra $f_i + \RR_+ A_i$ for which $f_i \in \lambda + L$.  By
Lemma~\ref{l:polyhedra}, it suffices to show that if $W$ is the
intersection of~$L$ with the relative interior of a polyhedron, then
$W$ possesses an affine stratification in which every stratum is a
translate of a normal affine semigroup.  Replacing $L$ with~$\ZZ^d$,
we may as well assume that $L = \ZZ^d$.  Lemma~\ref{l:module} implies
that $W$ is a finitely generated module over a normal affine
semigroup~$A$.  Thus the vector space over~$\CC$ with basis~$W$
constitutes a finitely generated $\ZZ^d$-graded submodule $M \subseteq
\CC[\ZZ^d]$ over the affine semigroup ring~$\CC[A]$.  The result now
follows by a simpler version of the argument in the proof of
Lemma~\ref{l:sat}, using a toric filtration: in the present case,
dimension is not an issue, and $M_j/M_{j-1}$ is already a
$\ZZ^d$-graded translate of a normal affine semigroup ring, because
every face of~$A$ is normal.
\end{proof}

\begin{cor}\label{c:affine}
Fix a linear map $\varphi: \ZZ^n \to \ZZ^d$.  If\/ $\cW \subseteq
\ZZ^n$ possesses an affine stratification then so does $\varphi(\cW)
\subseteq \ZZ^d$.
\end{cor}
\begin{proof}
The image of any translate of an affine semigroup in~$\ZZ^n$ is a
translate of an affine semigroup in~$\ZZ^d$, so use a stratification
of~$\cW$ as in Theorem~\ref{t:equiv}.6: nothing guarantees that the
images of the strata are disjoint, but that is irrelevant by
Theorem~\ref{t:equiv}.3.
\end{proof}

\begin{cor}\label{c:union}
If each of finitely many given subsets of\/~$\ZZ^d$ possesses an
affine stratification, then so does their union.
\end{cor}
\begin{proof}
Use the equivalence of (for example) Theorem~\ref{t:equiv}.1 and
Theorem~\ref{t:equiv}.2.
\end{proof}

\begin{remark}
The reason for choosing Definition~\ref{d:strat} as the fundamental
concept instead of the other conditions in Theorem~\ref{t:equiv} is
that Definition~\ref{d:strat} likely results in the most efficient
data structure for algorithmic purposes; see \cite{algsCGT}.
\end{remark}

\section{Fibers of affine presentations of commutative monoids}\label{s:fibers}

This section serves as an elementary example of the theory of
mesoprimary decomposition and as a bridge to combinatorial game theory
in the presence of finite mis\`ere quotients considered in
Section~\ref{s:misere}.  The main observation in this section is that
affine stratifications exist for fibers of affine presentations of
arbitrary commutative monoids.  The relevant statement in
Theorem~\ref{t:strat} requires no additional background, but its more
detailed final claim invokes notions from \emph{mesoprimary
decomposition} \cite{mesoprimary}.  The prerequisites for the
statement are \cite[Definitions~2.11 and 5.2]{mesoprimary}; the proofs
also need \cite[Definitions~2.8, 3.10, 6.2, 6.5, and~8.1,
Corollary~6.6, and Theorem~8.3]{mesoprimary}.

\begin{lemma}\label{l:L}
To specify a \emph{prime congruence}
\cite[Definition~2.11]{mesoprimary} on an affine~semigroup it is
equivalent to pick a face~$F$ and subgroup $L \subseteq \ZZ F$ of the
universal group~of~$F$.
\end{lemma}
\begin{proof}
The face~$F$ is the set of elements of the ambient affine
semigroup~$A$ whose images are not \emph{nil}
\cite[Definition~2.8]{mesoprimary}.  The subgroup $L$ is the kernel of
the group homomorphism obtained by \emph{localizing} the monoid
morphism $F \to A/\til$ along~$F$ \cite[Definition~3.10]{mesoprimary},
where $\til$ is the prime congruence in question.
\end{proof}

\begin{defn}\label{d:assoc}
A lattice $L \subseteq \ZZ A$ is \emph{associated} to a
congruence~$\til$ on an affine semigroup~$A$ if~$\til$ has an
associated prime congruence \cite[Definition~5.2]{mesoprimary}
specified by $L \subseteq \ZZ F$ for some face~$F$ as in
Lemma~\ref{l:L}, and in this case $F$ is an \emph{associated face}
of~$\til$.
\end{defn}

\begin{thm}\label{t:strat}
If\/ $\varphi: A \onto Q$ is a surjection of a commutative monoids
with $A$ an affine semigroup, then every fiber of~$\varphi$ possesses
an affine stratification.  Moreover, the stratification can be chosen
so that each of its affine semigroups is $L \cap A$ for some
intersection $L$ of associated lattices of the congruence
defining~$\varphi$.
\end{thm}

The proof, included after Proposition~\ref{p:Lbounded}, requires some
preliminary results on affine stratifications in simpler situations
than Theorem~\ref{t:strat}.

\begin{lemma}\label{l:stanley}
If $M$ is an ideal in an affine semigroup~$A$, then $M$ possesses an
affine stratification in which each stratum is a translate $f_i + A_i$
of a face $A_i$ of~$A$.
\end{lemma}
\begin{proof}
Use a toric filtration as in the proof of Lemma~\ref{l:sat}, where $M$
there is replaced by the $\CC[A]$-module $\CC\{M\}$ that has the ideal
$M \subseteq A$ as a vector space basis over~$\CC$.
\end{proof}

\begin{example}\label{e:stanley}
Every ideal in $\NN^n$ has a \emph{Stanley decomposition}
\cite{Sta82}: an expression as a finite disjoint union of translates
of faces~$\NN^J$ of\/~$\NN^n$ for $J \subseteq \{1,\ldots,n\}$; see
\mbox{\cite[\S2]{cmMonomial}}.
\end{example}

\begin{lemma}\label{l:ideal}
Fix a normal affine semigroup~$A$.  The intersection $(\uu + L) \cap
M$ of any coset of a lattice $L \subseteq \ZZ A$ with an ideal $M
\subseteq A$ is a finitely generated module for~$L \cap A$.
\end{lemma}
\begin{proof}
The ideal $M$ is finitely generated, and the intersection $(\uu + L)
\cap (\aa + A)$ is finitely generated as a module over $L \cap A$ for
any $\aa \in A$ by Lemma~\ref{l:module}.
\end{proof}

\begin{prop}\label{p:Lbounded}
If $L$ is the associated lattice of a mesoprimary congruence on an
affine semigroup~$A$, then every congruence class is a finite union of
sets $(\aa + L) \cap A$.
\end{prop}
\begin{proof}
Let $F$ be the associated face, and write $Q$ for the quotient of~$A$
modulo the mesoprimary congruence.  Since the localization morphism $Q
\to Q_F$ along~$F$ is injective, it suffices to show that every
element of~$Q_F$, each viewed as a subset of the localization $A_F = A
+ \ZZ F$, is a finite union of cosets of~$L$.  Under the action
of~$\ZZ F$ on~$Q_F$, the stabilizer of every non-nil element $q \in
Q_F$ is~$L$; this is what semifreeness means in
\cite[Corollary~6.6]{mesoprimary}.  Consequently, viewing $q$ as a
subset of~$A_F$, the intersection of~$q$ with any single coset of~$\ZZ
F$ is a single coset of~$L$.  Hence the result follows from finiteness
of the number of $F$-orbits \cite[Corollary~6.6]{mesoprimary}.
\end{proof}

\begin{proof}[Proof of Theorem~\ref{t:strat}]
Let $\til$ be the congruence on~$A$ induced by~$\varphi$.  The fiber
over the nil of~$Q$, if there is one, is an ideal of~$A$, which has an
affine stratification by Lemma~\ref{l:stanley}.  To treat the the
non-nil fibers, fix a mesoprimary decomposition of~$\til$
\cite[Theorem~8.3]{mesoprimary}.  Let $L_i \subseteq \ZZ F_i$ for $i =
1,\ldots,r$ be the lattices associated to~$\til$ with corresponding
associated faces $F_i$ of~$A$.  Write $\oQ_i$ for the quotient of~$A$
modulo the $i^\th$ mesoprimary congruence in the decomposition.

Every class $q \in Q$ is, as a subset of~$A$, the intersection of the
$r$ mesoprimary classes $\ol q_i \in \oQ_i$ containing~$q$, because
$\til$ is the common refinement of its mesoprimary components.
Furthermore, as long as $q$ is not nil, at least one of the
mesoprimary classes~$\ol q_i$ is not nil.  For any such non-nil
class~$\ol q_i$, Proposition~\ref{p:Lbounded} guarantees a finite set
$U_i$ such that $\ol q_i = \bigcup_{\uu \in U_i} (\uu + L_i) \cap A$.
Renumbering for convenience, assume that $\ol q_i$ is non-nil for $i
\leq k$ and nil for $i > k$.  Then $q' := \ol q_1 \cap \cdots \cap \ol
q_k$ is the intersection with~$A$ of a finite union of cosets of the
lattice $L_1 \cap \cdots \cap L_k$, and $q'' := \ol q_{k+1} \cap
\cdots \cap \ol q_r$ is an ideal of~$A$.  Since $q = q' \cap q''$, the
result follows from Lemma~\ref{l:ideal}.
\end{proof}

\section{Lattice games and mis\`ere quotients}\label{s:misere}

Fix a pointed rational convex polyhedron $\Pi \subseteq \RR^d$ with
recession cone~$C$ of dimension~$d$.  The pointed hypothesis means
that $\Pi = P + C$ for some polytope (i.e., bounded convex
polyhedron)~$P$.  Write $\Lambda = \Pi \cap \ZZ^d$ for the set of
integer points in~$\Pi$.  The following definitions summarize
\cite[Definition~2.3, Definition~2.9, and Lemma~3.5]{latticeGames}.

\begin{defn}\label{d:ruleset}
A finite subset $\Gamma \subset \ZZ^d \minus \{\0\}$ is a \emph{rule
set} if
\begin{enumerate}
\item%
there exists a linear function on $\RR^d$ that is positive on
$\Gamma \cup C \minus \{\0\}$; and
\item%
there is finite set $F \subset \Lambda$ such that every position $p
\in \Lambda$ has a \emph{$\Gamma$-path in~$\Lambda$ to~$F$}: a
sequence $p = p_r, \ldots, p_0 \in \Lambda$ with $p_0 \in F$ and $p_k
- p_{k-1} \in \Gamma$ for $k = \{1,\ldots,r\}$.
\end{enumerate}
\end{defn}

For the next definition, it is important to observe that the rule set
$\Gamma$ induces a partial order $\preceq$ on~$\Lambda$ in which $p
\preceq q$ if $q - p$ lies in the monoid $\NN\Gamma$ generated
by~$\Gamma$ \cite[Lemma~2.8]{latticeGames}.  An \emph{order ideal}
under this (or any) partial order~$\preceq$ is a subset $S$ closed
under going down: $p \preceq q$ and $q \in S \implies p \in S$.

\begin{defn}\label{d:latticegame}
Fix a rule set $\Gamma$.
\begin{itemize}
\item%
A \emph{game board} $B$ is the complement in $\Lambda$ of a finite
$\Gamma$-order ideal in $\Lambda$ called the set of \emph{defeated
positions}.
\item%
A \emph{lattice game} $G = (\Gamma,B)$ is defined by a rule
set~$\Gamma$ and a game board~$B$.
\item%
The \emph{$P$-positions} of~$G$ form a subset $\cP \subseteq B$ such
that $(\cP + \Gamma) \cap B = B \minus \cP$.
\item%
An \emph{affine stratification} of~$G$ is an affine stratification of
its set of $P$-positions.
\end{itemize}
\end{defn}

The $P$-positions of~$G$ are uniquely determined by the rule set and
game board \cite[Theorem~4.6]{latticeGames}.

Following Plambeck and Siegel \cite{Pla05,misereQuots}, every lattice
game possesses a unique quotient that optimally collapses $\Lambda$
while faithfully recording the interaction of the $P$-positions with
its additive structure.

\begin{defn}\label{d:cong}
Two positions $p,q \in B$ are \emph{indistinguishable}, written $p
\sim q$, if
$$%
  (p + C) \cap \cP = p - q + (q + C) \cap \cP.
$$
In other words, $p + r \in \cP \iff q + r \in \cP$ for all $r$ in the
recession cone~$C$ of~$B$.  The \emph{mis\`ere quotient} of the
lattice game with winning positions~$\cP$ is the quotient
$\Lambda/\til$ of the polyhedral set~$\Lambda$ modulo
indistinguishability.
\end{defn}

Geometrically, indistinguishability means that the P-positions in the
cone above~$p$ are the same as those above~$q$, up to translation by
$p-q$.  It is elementary to verify that indistinguishability is an
equivalence relation, and that it is additive, in the sense that $p
\cong q \implies p + r \cong q + r$ for all $r \in C \cap \ZZ^d$.
Thus, when $B = \Lambda = C \cap \ZZ^d$ is a monoid,
indistinguishability is a congruence, so the quotient of~$B$ modulo
indistinguishability is again a monoid.

\begin{lemma}
Every fiber of the projection $\Lambda \to \Lambda/\til$ either
consists of $P$-positions or has empty intersection with~$\cP$.
\end{lemma}
\begin{proof}
If $p \sim q$, then by definition either $p$ and~$q$ are both
$P$-positions or neither is.
\end{proof}

\begin{cor}\label{c:strat}
Fix a lattice game $G = (\Gamma,B)$ played on a cone, meaning that
the game board is the complement of the defeated positions in a normal
affine semigroup~$\Lambda$.  If the mis\`ere quotient $\Lambda/\til$
is finite, then $G$ admits an affine stratification.
\end{cor}
\begin{proof}
The set~$\cP$ of $P$-positions is a union of fibers of the projection.
If the quotient is finite, then the union is finite.  Now apply
Theorem~\ref{t:strat}.
\end{proof}

Although it is useful to record Corollary~\ref{c:strat}, which treats
the case of finite mis\`ere quotient \emph{monoids}, where the game is
played on a cone, the extension to arbitrary finite mis\`ere quotients
requires little additional work.

\begin{thm}\label{t:affineStrat}
Fix a lattice game $G = (\Gamma,B)$ played on a polyhedral set
$\Lambda = \Pi \cap \ZZ^d$.  If the mis\`ere quotient $\Lambda/\til$
is finite, then $G$ admits an affine stratification.
\end{thm}
\begin{proof}
The set~$\cP$ of $P$-positions is a union of fibers of the projection
$\Lambda \to \Lambda/\til$.  Since the quotient is finite, the union
is finite.  Therefore it suffices to show that every fiber $\Phi
\subseteq \Lambda$ of the projection possesses an affine
stratification.

The pointed hypothesis on~$\Pi$ implies that $\Lambda = F + A$, where
$F \subseteq \ZZ^d$ is finite and $A = C \cap \ZZ^d$ is a normal
affine semigroup.  The fiber $\Phi$ is a finite union $\Phi =
\bigcup_{f \in F} \Phi\cap(f+A)$.  By Corollary~\ref{c:union}, it
therefore suffices to show that for each lattice point $f \in F$,
every fiber of the map $f + A \to \Lambda/\til$ possesses an affine
stratification.  But the composite map $A \to f + A \to \Lambda/\til$
induces a congruence on~$A$ whose classes are the fibers, to which
Theorem~\ref{t:strat} applies.
\end{proof}



\end{document}